\documentclass{amsart}
\usepackage{latexsym,amsxtra,amscd,ifthen}
\usepackage{amsfonts}
\usepackage{verbatim}
\usepackage{amsmath}
\usepackage{amsthm}
\usepackage{amssymb}
\usepackage{url}
\usepackage[arrow,matrix]{xy}
\usepackage[colorlinks=true,citecolor=blue,linkcolor=blue]{hyperref}

\theoremstyle{plain}

\newtheorem{theorem}{Theorem}
\newtheorem{lemma}[theorem]{Lemma}
\newtheorem{proposition}[theorem]{Proposition}

\numberwithin{theorem}{section}
\numberwithin{equation}{theorem}

\theoremstyle{definition}
\newtheorem{definition}[theorem]{Definition}
\newtheorem{setup}[theorem]{Setup}
\newtheorem{example}[theorem]{Example}
\newtheorem{remark}[theorem]{Remark}

\newtheorem*{question*}{Question}

\DeclareMathOperator{\ch}{char}

\DeclareMathOperator{\Ext}{Ext}
\DeclareMathOperator{\Tor}{Tor}
\DeclareMathOperator{\Hom}{Hom}
\DeclareMathOperator{\pd}{projdim}

\DeclareMathOperator{\im}{im}
\DeclareMathOperator{\depth}{depth}
\DeclareMathOperator{\Aut}{Aut}

\DeclareMathOperator{\Mod}{Mod}

\def\mfm{\mathfrak m}
\def\kk{\operatorname{\bf k}}

\def\QMod{\operatorname{QMod}}

\def\ch{\operatorname{char}}

\begin{document}

\title[Cohen-Macaulay invariant subalgebras]{Cohen-Macaulay invariant subalgebras of Hopf dense Galois extensions}

\author{Ji-Wei He and Yinhuo Zhang}
\address{He: Department of Mathematics,
Hangzhou Normal University,
Hangzhou Zhejiang 310036, China}
\email{jwhe@hznu.edu.cn}
\address{Zhang:
Department of Mathematics and Statistics, University of Hasselt, Universitaire Campus,
3590 Diepenbeek, Belgium} \email{yinhuo.zhang@uhasselt.be}

\begin{abstract} Let $H$ be a semisimple Hopf algebra, and let $R$ be a noetherian left $H$-module algebra. If $R/R^H$ is a right $H^*$-dense Galois extension, then the invariant subalgebra $R^H$ will inherit the AS-Cohen-Macaulay property from $R$ under some mild conditions, and $R$, when viewed as a right $R^H$-module, is a Cohen-Macaulay module. In particular, we show that if $R$ is a noetherian complete semilocal algebra which is AS-regular of global dimension 2 and $H=\kk G$ for some finite subgroup $G\subseteq\Aut(R)$, then all the indecomposable Cohen-Macaulay module of $R^H$ is a direct summand of $R_{R^H}$, and hence $R^H$ is Cohen-Macaulay-finite, which generalizes a classical result for commutative rings. The main tool used in the paper is the extension groups of objects in the corresponding quotient categories.
\end{abstract}





\maketitle


\setcounter{section}{-1}
\section{Introduction}

Motivated by the study of quotient singularities of noncommutative projective schemes, Van Oystaeyen and the authors introduced the concept of Hopf dense Galois extension in \cite{HVZ2}. The present paper is a further study of Hopf dense Galois extensions. In this paper, we focus on the Cohen-Macaulay properties of the invariant subalgebra of a Hopf dense Galois extension. Let $H$ be a finite dimensional semisimple Hopf algebra, and let $R$ be a left $H$-module algebra. Let $H^*$ be the dual Hopf algebra. Since $H$ is finite dimensional, $R$ is a right $H^*$-comodule algebra. Denote by $R^H(=R^{coH^*})$ the invariant subalgebra. The algebra extension $R/R^H$ is called a {\it right $H^*$-dense Galois extension} (\cite[Definition 1.1]{HVZ2}) if the cokernel of the map
$$\beta:R\otimes_{R^H} R\to R\otimes H^*, r\otimes r'\mapsto (r\otimes1)\rho(r')$$ is finite dimensional, where $\rho:R\to R\otimes H^*$ is the right $H^*$-comodule structure map.

The concept of a Hopf dense Galois extension is a weaker version of that of a Hopf Galois extension.
The main feature of a Hopf dense Galois extension is that there is an equivalence between some quotient categories. More precisely, let $R$ be a noetheiran algebra which is a left $H$-module algebra over a finite dimensional semisimple Hopf algebra $H$. Let $\Mod R$ be the category of right $R$-modules and let $\Tor R$ be the full subcategory of $\Mod R$ consisting of torsion $R$-modules (see Section \ref{sec1}). Then we obtain a quotient category $\QMod R=\frac{\Mod R}{\Tor R}$. Set $B=R\#H$. Since $R$ is noetherian, $B$ and the invariant subalgebra $R^H$ are also noetherian. Then we also have quotient category $\QMod B$ and $\QMod R^H$. If $R/R^H$ is right $H^*$-dense Galois extension, then there is an equivalence of abelian categories $\QMod R^H\cong\QMod B$ (cf. \cite[Theorem 2.4]{HVZ2}). Under this equivalence, we show in this paper that the invariant subalgebra $R^H$ inherits many homological properties from $R$. Especially, we have the following results (Theorems \ref{thm1} and \ref{thm4.1}). The terminologies may be found in Section \ref{sec3}. \\

\noindent{\bf Theorem.} {\it Let $H$ be a finite dimensional semisimple Hopf algebra, and let $R$ be a noetherian left $H$-module algebra. Assume that the $H$-module $R$ is admissible and $\Tor B$ is stable. If $R/R^H$ is a right $H^*$-dense Galois extension, then the following statements hold.

{\rm(i)} If $R$ is AS-Cohen-Macaulay of dimension $d$, so is $R^H$.

{\rm(ii)} If $R$ is AS-Cohen-Macaulay of dimension $d$, then $R$, viewed as a right $R^H$-module, is Cohen-Macaulay.}\\

There are plenty of $H$-module algebras satisfying the assumptions in the theorem above. For example, if $R$ is a noetherian complete semilocal algebra and $H$ is a finite group algebra, then the $H$-module algebra $R$ is admissible and $\Tor B$ is stable (cf. Example \ref{ex1}).

The main tool for proving of the theorem above is the following observation (Theorem \ref{thm0}) which gives a way to compute the extension groups of objects in the quotient category $\QMod R^H$.\\

\noindent{\bf Theorem.} {\it Assume that the $H$-module algebra $R$ is admissible and $\Tor B$ is stable. Let $N$ and $M$ be finitely generated right $B$-modules. Let $\mathcal{N}$ and $\mathcal{M}$ be the corresponding objects of $N$ and $M$ in the quotient categories. Then we have

{\rm (i)} for each $i\ge0$, $\Ext^i_{\QMod R}(\mathcal{N},\mathcal{M})$ is a right $H$-module,

{\rm(ii)} there are isomorphisms $\Ext^i_{\QMod B}(\mathcal{N},\mathcal{M})\cong\Ext^i_{\QMod R}(\mathcal{N},\mathcal{M})^H,\ \forall\ i\ge0.$}\\

If $R$ is further a noetherian complete semilocal algebra which is also AS-regular of global dimension 2, and $G\subseteq \Aut(R)$ is a finite subgroup, then the invariant subalgebra $R^G$ is Cohen-Macaulay-finite. Indeed, we have the following result (Theorem \ref{thm5.1}), which is a noncommutative version of a classical result for commutative rings due to Auslander \cite{A}.\\

\noindent{\bf Theorem.} {\it Let $R$ be a noetherian complete semilocal algebra, and let $G\subseteq\Aut(R)$ be a finite subgroup. Assume $\ch\kk\nmid|G|$ and $R/R^G$ is a right $\kk G^*$-dense Galois extension and $R$ is AS-regular of global dimension 2. Then the following statements hold.

{\rm(i)} $R^G$ is AS-Cohen-Macaulay of dimension 2.

{\rm(ii)} A finitely generated right $R^G$-module $M$ is Cohen-Macaulay if and only if $M\in \text{\rm add} (R_{R^G})$, where $\text{\rm add} (R_{R^G})$ is the subcategory of $\Mod R^G$ consisting of all the direct summands of finite direct sums of $R_{R^G}$.}\\

When the algebra $R$ is graded, the above result was obtained by several authors (see Remark \ref{rm5.1}). However, the methods applied in graded case can not be applied to nongraded case directly. Our method depends on the relations of extension groups of objects in the quotient categories over $R^G$ and $R\#\kk G$ respectively, which is different from the methods applies in the graded case.

Throughout, $\kk$ is a field. All the algebras and modules are over $\kk$. Unadorned $\otimes$ means $\otimes_{\kk}$. We refer to \cite{HVZ2} for basic properties of Hopf dense Galois extensions.

\section{Hom-sets of the quotient category}\label{sec1}

Let $B$ be a (right) noetherian algebra. We denote by $\Mod B$ the category of right $B$-modules. Let $M$ be a right $B$-module. We say that a submodule $K$ of $M$ is {\it cofinite} if $M/K$ is finite dimensional. For an element $m\in M$, we say that $m$ is an torsion element if $mB$ is finite dimensional. Let $\Gamma_B (M)=\{m\in M|m\text{ is a torsion element}\}$. Then $\Gamma_B(M)$ is a submodule of $M$. Note that $\Gamma_B$ is in fact a functor from $\Mod B$ to itself. The functor $\Gamma_B$ can be represented as $\Gamma_B=\underset{\longrightarrow}\lim\Hom_B(B/K,-)$, where $K$ runs over all the cofinite right ideals of $B$. In the sequel, if there is no risk of confusion, we will omit the subscript in the functor $\Gamma_B$, and simply write as $\Gamma$.

For a right $B$-module, if $\Gamma(M)=0$, then we say that $M$ is {\it torsion free}. If $\Gamma(M)=M$, then $M$ is called a {\it torsion module}.
Let $\Tor B$ be the full subcategory of $\Mod B$ consisting of torsion modules. One sees that $\Tor B$ is a localizing subcategory of $\Mod B$. Hence we have an abelian quotient category $\QMod B:=\frac{\Mod B}{\Tor B}$.
The natural projection functor is denoted by $\pi:\Mod B\to\QMod B$, which has a right adjoint functor $\omega:\QMod B\to \Mod B$. For $N,M\in \Mod B$, we write $\mathcal{N}$ and $\mathcal{M}$ for $\pi(N)$ and $\pi(M)$ respectively in the quotient category. The morphisms in $\QMod B$ is defined to be the set
$$\Hom_{\Mod B}(\mathcal{N},\mathcal{M})=\underset{\longrightarrow}\lim\Hom_B(N',M/\Gamma(M)),$$ where $N'$ runs over all the cofinite submodules of $N$, and the direct system is induced by the inclusion maps of the cofinite submodules.

If $M$ is a torsion free module, then its injective envelope is torsion free as well. The injective envelope of a torsion module is not torsion in general. If $\Tor B$ is closed under taking injective envelope, then we say that $\Tor B$ is {\it stable}. For example, if $B$ is a noetherian commutative local algebra, then $B$ has a stable torsion class.

\begin{lemma} \label{lem0} Let $N$ be a finitely generated right $B$-module, and let $E$ be an injective torsion $B$-module. Then $\underset{\longrightarrow}\lim\Hom_B(N',E)=0$, where $N'$ runs over all the cofinite submodule of $N$.
\end{lemma}
\begin{proof} For every cofinite submodule $N'$ of $N$, applying the functor $\Hom_B(-,E)$ to the exact sequence $0\to N'\to N\to N/N'\to0$, we obtain the exact sequence $0\to \Hom_B(N/N',E)\to\Hom_B(N,E)\to\Hom_B(N',E)\to0$. Take the direct limit over all the cofinite submodules $N'\subseteq N$, we obtain exact sequence $$0\to \underset{\longrightarrow}\lim\Hom_B(N/N',E)\to\Hom_B(N,E)\to\underset{\longrightarrow}\lim\Hom_B(N',E)\to0.$$ For each $B$-module morphism $f:N\to E$, since $N$ is finitely generated, $\im f$ is finite dimensional. Hence $\ker f$ is cofinite. Since $f$ factors through $N/\ker f$, we see that the morphism $\underset{\longrightarrow}\lim\Hom_B(N/N',E)\to\Hom_B(N,E)$ is an epimorphism, and hence $\underset{\longrightarrow}\lim\Hom_B(N',E)=0$.
\end{proof}

\begin{lemma}\label{lem1} Let $B$ be a noetherian algebra. Assume that $\Tor B$ is stable. For every finitely generated module $N,M\in \Mod B$, if $N$ is finitely generated, then $\Ext^i_{\QMod B}(\mathcal{N},\mathcal{M})=\underset{\longrightarrow}\lim\Ext^i_B(N',M)$ for all $i\ge0$, where $N'$ runs over all the cofinite submodules of $N$.
\end{lemma}
\begin{proof} Take a minimal injective resolution of $M$ as follows
\begin{equation}\label{eq1}
  0\to M\to I^0\to I^1\to\cdots\to I^n\to\cdots.
\end{equation}
Since $\Tor B$ is stable, for each $n\ge0$, $I^n=F^n\oplus E^n$ where $F^n$ is a torsion free module and $E^n$ is a torsion module. Note that if $F$ is an injective torsion free module, then $\mathcal{F}=\pi(F)$ is an injective object in $\QMod B$. Applying the projection functor to the exact sequence (\ref{eq1}), we obtain an injective resolution of $\mathcal{M}$ in $\QMod B$:
\begin{equation}\label{eq2}
  0\to \mathcal{M}\to \mathcal{F}^0\to \mathcal{F}^1\to\cdots\to \mathcal{F}^n\to\cdots.
\end{equation}
Applying the functor $\Hom_{\QMod B}(\mathcal{N},-)$ to (\ref{eq2}), we have
$$0\to\Hom_{\QMod B}(\mathcal{N},\mathcal{F}^0)\to\cdots\to \Hom_{\QMod B}(\mathcal{N},\mathcal{F}^n)\to\cdots,$$
which is equivalent to the following complex
\begin{equation}\label{eq3}
  0\to\underset{\longrightarrow}{\lim}{\Hom}_{B}(N',F^0)\to\cdots\to \underset{\longrightarrow}{\lim}{\Hom}_{B}(N',F^n)\to\cdots.
\end{equation}
By Lemma \ref{lem0}, the complex (\ref{eq3}) is equivalent to
\begin{equation}\label{eq4}
  0\to\underset{\longrightarrow}{\lim}\Hom_{B}(N',F^0\oplus E^n)\to\cdots\to \underset{\longrightarrow}{\lim}{\Hom}_{B}(N',F^n\oplus E^n)\to\cdots.
\end{equation}
Taking the cohomology of (\ref{eq4}), we obtain the desired result since the direct limit is exact.
\end{proof}

\section{Hom-sets of the quotient categories of smash products}\label{sec2}

Throughout this section, let $R$ be a noetherian algebra, $H$ a finite dimensional semisimple Hopf algebra acting on $R$ so that $R$ is a left $H$-module algebra. Let $B=R\# H$ be the smash product. We may view $R$ as a left $B$-module by setting the left $B$-action 
\begin{equation}\label{eq2.1}
  (a\#h)\cdot r=a(h\cdot r),\ \forall\ a,r\in R, h\in H,
\end{equation}
where $h\cdot r$ is the left $H$-action on $R$. Similarly, we may view $R$ as a right $B$-module by setting the right $B$-action 
\begin{equation}\label{eq2.2}
  r\cdot (a\#h)=(S^{-1}h_{(1)}\cdot r)(h_{(2)}\cdot a),\ \forall\ a,r\in R, h\in H,
\end{equation}
where $S$ is the antipode of $H$ and $\Delta(h)=h_{(1)}\otimes h_{(2)}$.

Let $M$ be a right $B$-module. For every element $x\in M$, we see that $xB$ is finite dimensional if and only if $xR$ is finite dimensional since $H$ is finite dimensional. Hence we see as a right $B$-module $\Gamma_R(M)=\Gamma_B(M)$. From this observation, we obtain the following fact.

\begin{lemma}\label{lemx1} Let $M$ be a right $B$-module. Then as right $R$-modules, $R^i\Gamma_R(M)\cong R^i\Gamma_B(M)$ for all $i\ge0$.
\end{lemma}
\begin{proof} Let $0\to M\to I^0\to I^2\to \cdots$ be an injective resolution of $M$ in the category of right $B$-modules. By \cite[Proposition 2.6]{HVZ}, each $I^i$ is injective as a right $R$-module. Hence $R^i\Gamma_R(M)$ is the $i$th cohomology of the complex $\Gamma_R(I^\cdot)$. Since $\Gamma_R(I^i)=\Gamma_B(M)$ for all $i\ge0$, as complexes of right $R$-modules, $\Gamma_R(I^\cdot)=\Gamma_B(I^\cdot)$. Thus we have $R^i\Gamma_R(M)=R^i\Gamma_B(M)$.
\end{proof}

Let $N$ be a right $R$-module. We define a right $B$-module $N\# H=N\otimes H$ whose right $B$-module action is given by $(n\otimes h)(a\otimes g)=n(h_{(1)}a)\otimes h_{(2)}g$ for $n\in N, g,h\in H$.

\begin{lemma}\label{lem2} The following statements hold.

{\rm(i)} If $\Tor B$ is stable, then so is $\Tor R$.

{\rm(ii)} Let $A$ and $A'$ be two noetherian algebras. If $F:\Mod A\longrightarrow\Mod A'$ is an equivalence of abelian categories, then the restriction of $F$ to $\Tor A$ induces an equivalence $\Tor A\longrightarrow\Tor A'$.
\end{lemma}
\begin{proof} (i) Let $N$ be a torsion $R$-module. Set $M=N\# H$. Then $M$ is $B$-torsion module. Let $I$ be the injective envelope of $M$. Since $\Tor B$ is stable, $I$ is a torsion $B$-module. By \cite[Proposition 2.6]{HVZ}, $I$ is injective as a right $R$-module. We view $N$ as an $R$-submodule of $M$ through the inclusion map $N\to M: n\mapsto n\#1$. Then $N$ is a submodule of the torsion injective module $I$. Hence the injective envelope of $N$ must be torsion.

(ii) This is classical.
\end{proof}

\begin{proposition}\label{prop1} Let $H$ be a finite dimensional semisimple and cosemisimple Hopf algebra. Let $R$ be a noetherian left $H$-module algebra. Then $\Tor R$ is stable if and only if $\Tor B$ is stable.
\end{proposition}
\begin{proof} We only need to prove the ``only if'' part. Note that $B=R\# H$ is a right $H$-comodule algebra. It is a left $H^*$-module algebra. The invariant subalgebra $B^{H^*}=R$. Since $B$ is a right $H$-Galois extension of $R$, $R$ and $B\#H^*$ are Morita equivalent (cf. \cite[Theorem 1.2]{CFM}). Since $\Tor R$ is stable, $\Tor B\#H^*$ is stable by Lemma \ref{lem2}(ii). By Lemma \ref{lem2}(i), $\Tor B$ is stable.
\end{proof}

Let $N$ and $M$ be right $B$-modules. Note that $\Hom_R(N,M)$ has a right $H$-module action defined as follows: for $f\in \Hom_R(N,M), h\in H$,
\begin{equation}\label{eq5}
  (f\leftharpoonup h)(n)=f(nS(h_{(1)}))h_{(2)},\quad \forall\ n\in N.
\end{equation}
Under this $H$-action on $\Hom_R(N,M)$, we have the following isomorphism \cite{CFM}
\begin{equation}\label{eq6}
  \Hom_B(N,M)\cong \Hom_R(N,M)^H.
\end{equation}
The right $H$-action on $\Hom_R(N,M)$ may be extended to the extension groups by choosing a projective resolution of $N$ (or injective resolution of $M$) in $\Mod B$. Moreover, we have the following isomorphisms \cite{HVZ}
\begin{equation}\label{eq7}
  \Ext^i_B(N,M)\cong \Ext^i_R(N,M)^H, \ \forall\ i\ge0.
\end{equation}

\begin{definition}\label{def2.1} We say that the left $H$-module algebra $R$ is (right) {\it admissible} if for any finitely generated right $B$-module $N$, and any cofinite $R$-submodule $K$ of $N$, there is a cofinite $B$-submodule $K'$ of $N$ such that $K'\subseteq K$.
\end{definition}

There are several natural classes of noetherian admissible noetherian $H$-module algebras.

\begin{example}\label{ex1} Let $R$ be a noetherian semilocal algebra, that is, $R/J$ is a finite dimensional semisimple algebra where $J$ is the Jacobson radical of $R$. Let $G$ be a finite group which acts on $R$ so that $R$ is a left $G$-module algebra. Let $B=R\#\kk G$. We know that $J$ is stable under the $G$-action. Let $N$ be a finitely generated right $B$-module, and let $K$ be an $R$-submodule such that $\overline{N}:=N/K$ is finite dimensional. Assume $\overline{N}\neq0$. Note that $\overline{N}J\neq \overline{N}$. Since $\overline{N}$ is finite dimensional, there is an integer $k$ such that $\overline{N}J^k=0$. Hence $NJ^k\subseteq K$. For $x\in N$, $r\in J^k$ and $g\in G$, we have $(xr)g=(xg)(g^{-1}(r))\in NJ^k$. Hence $NJ^k$ is a $B$-submodule of $N$. Since $N/(NJ^k)$ is a finitely generated module over $R/J^k$, it is finite dimensional. Hence the $G$-module algebra $R$ is admissible.

If $R$ is also complete with respect to the $J$-adic topology and $\ch \kk\nmid |G|$, then $\Tor B$ is stable. Indeed, a right $B$-module $M$ is torsion if and only if $M$ is torsion as a right $R$-module. Since $R$ is complete, by \cite[Theorem 1.1]{J}, $J$ has the Atin-Rees property which insures that the injective envelope of a torsion $R$-module is still torsion. Since $\kk G$ is both semisimple and cosemisimple, $\Tor B$ is stable by Proposition \ref{prop1}.
\end{example}

\begin{remark} In Definition \ref{def2.1}, if $R$ is a graded algebra and the $H$-action is homogeneous, then we say that the left graded $H$-module algebra $R$ is admissible if the above conditions hold for finitely generated graded modules.
\end{remark}

\begin{example} \label{ex2} Let $R=R_0\oplus R_1\oplus \cdots$ be a graded noetherian algebra such that $\dim R_i<\infty$ for all $i\ge0$. Let $H$ be a finite dimensional semisimple Hopf algebra which acts on $R$ homogeneously so that $R$ is an $H$-module algebra. Set $B=R\#H$. Let $N$ be a graded finitely generated right $B$-module, and $K$ a graded $R$-submodule of $N$ such that $N/K$ is finite dimensional. Let $\mfm=\oplus_{i\ge1}R_i$. Since $N/K$ is finite dimensional, there is an integer $k$ such that $(N/K)\mfm^k=0$. Hence $N\mfm^k\subseteq K$. Since $\mfm$ is stable under the $H$-action, $N\mfm^k$ is a $B$-submodule of $N$. As $R$ is noetherian, $R/\mfm^k$ is finite dimensional for all $k\ge0$. Hence $N/(N\mfm^k)$ is finite dimensional.
It follows that the left $H$-module algebra $R$ is admissible in the graded sense.
\end{example}

\begin{lemma} \label{lem6} Let $R$ be an admissible $H$-module algebra. Let $N$ be a finitely generated right $B$-module, and $M$ a right $R$-module. Then $$\Hom_{\QMod R}(\mathcal{N},\mathcal{M})=\underset{\longrightarrow}{\lim}\Hom_R(K,M),$$ where $K$ runs over all the cofinite $B$-submodules of $N$.
\end{lemma}
\begin{proof} By Lemma \ref{lem1}, $\Hom_{\QMod R}(\mathcal{N},\mathcal{M})=\underset{\longrightarrow}{\lim}\Hom_R(K,M)$ where the limit runs over all the cofinite $R$-submodules $K$ of $N$. Since the $H$-module algebra $R$ is assumed to be admissible, the direct system formed by all the cofinite $R$-submodules of $N$ and the direct system formed by all the cofinite $B$-submodules of $N$ are cofinal. Hence $\underset{\longrightarrow}{\lim}\Hom_R(K,M)=\underset{\longrightarrow}{\lim}\Hom_R(K',M)$, where on the left hand the direct limit runs over all the cofinite $R$-submodules $K$ of $N$, and on the right hand the direct limit runs over all the cofinite $B$-submodules $K'$ of $N$.
\end{proof}

\begin{proposition} \label{prop2}  Let $R$ be an admissible $H$-module algebra. Let $N$ and $M$ be finitely generated right $B$-modules. Then there is a natural right $H$-module action on $\Hom_{\QMod R}(\mathcal{N},\mathcal{M})$. Moreover, with this $H$-module structure we have  $$\Hom_{\QMod B}(\mathcal{N},\mathcal{M})\cong\Hom_{\QMod R}(\mathcal{N},\mathcal{M})^H,$$ where the isomorphism is functorial in $M$.
\end{proposition}
\begin{proof} By Lemma \ref{lem6}, $\Hom_{\QMod R}(\mathcal{N},\mathcal{M})=\underset{\longrightarrow}{\lim}\Hom_R(K,M)$, where each $K$ in the direct system is a $B$-submodule of $N$. By (\ref{eq5}), $\Hom_R(K,M)$ is a right $H$-module for each cofinite $B$-submodule $K$. It is easy to see that the direct system is compatible with the right $H$-module actions. Hence $\Hom_{\QMod R}(\mathcal{N},\mathcal{M})$ is a right $H$-module.

For each $B$-submodule $K$ of $N$, we have $\Hom_B(K,M)\cong \Hom_R(K,M)^H$ by (\ref{eq6}). Since $H$ is semisimple, the functor $(\ )^H$ commutates with taking direct limits. Hence
\begin{eqnarray*}
\Hom_{\QMod B}(\mathcal{N},\mathcal{M})&=&\underset{\longrightarrow}{\lim}\Hom_B(K,M)\\
&\cong&\underset{\longrightarrow}{\lim}\Hom_R(K,M)^H\\
&\cong&(\underset{\longrightarrow}{\lim}\Hom_R(K,M))^H\\
&\cong&\Hom_{\QMod R}(\mathcal{N},\mathcal{M})^H.
\end{eqnarray*}
If $M'$ is another $B$-module and there is a right $B$-module morphism $f:M\to M'$, then it is easy to see that the direct systems and the right $H$-module structures on $\Hom_R(K,M)$ and $\Hom_R(K,M')$ are compatible with the morphism $f$. Hence the isomorphism is functorial in $M$.
\end{proof}

Next we show that the proposition above can be extended to extension groups, that is, there are natural $H$-module actions on extension groups $\Ext_{\QMod R}^i(\mathcal{N},\mathcal{M})$ if the left $H$-module $R$ is admissible and $\Tor B$ is stable.

For the rest of this section, the $H$-module algebra $R$ is admissible and $\Tor B$ is stable. Let $M$ be a finitely generated right $B$-module. Take a minimal injective resolution of $M$ as follows:
\begin{equation}\label{eq8}
  0\to M\to I^0\to I^1\to\cdots\to I^k\to\cdots.
\end{equation}
Since $\Tor B$ is stable, the injective envelope of a torsion module is still torsion. Then each injective module in the above sequence has a decomposition $I^i=E^i\oplus T^i$ such that $E^i$ is a torsion free $B$-module and $T^i$ is a torsion $B$-module. Write $I^\cdot$ for the complex obtained from (\ref{eq8}) by dropping $M$ on the left in the beginning. Similar to the discussions in \cite[Section 7, P.271]{AZ}, we see that there are complexes $E^\cdot$, $T^\cdot$ and a morphism $f:E^\cdot[-1]\to T^\cdot$ such that $I^\cdot=cone(f)$. Applying the projection functor $\pi:\Mod B\to \QMod B$ to the resolution (\ref{eq8}) and noticing that $E^i$ is torsion free injective for all $i\ge0$, we obtain an injective resolution of $\mathcal{M}$ in $\QMod B$:
\begin{equation}\label{eq9}
  0\to \mathcal{M}\to \mathcal{E}^0\to \mathcal{E}^1\to\cdots\to\mathcal{E}^i\to\cdots.
\end{equation}
Note that any right $B$-module is also a right $R$-module. Moreover, a right $B$-module $K$ is torsion free if and only if it is torsion free as an $R$-module. We see that $E^i$ is also torsion free when it is viewed as an $R$-module, and $T^i$ is torsion as an $R$-module for all $i\ge0$. By \cite[Proposition 2.6]{HVZ}, $E^i$ is also injective as a right $R$-module for each $i\ge0$. Hence we have the following observation.

\begin{lemma}\label{lem7} Assume that $\Tor B$ is stable. The sequence (\ref{eq9}) is an injective resolution of $\mathcal{M}$ in $\QMod B$. If $\mathcal{M}$ and $\mathcal{E}^i$ $(i\ge0)$ are viewed as objects in $\QMod R$ (by an abuse of the notions), then (\ref{eq9}) is also an injective resolution of $\mathcal{M}$ in $\QMod R$.
\end{lemma}

\begin{theorem} \label{thm0} Assume that $R$ is an admissible $H$-module algebra and $\Tor B$ is stable. Let $N$ and $M$ be finitely generated right $B$-modules. Then for each $i\ge0$, $\Ext^i_{\QMod R}(\mathcal{N},\mathcal{M})$ is a right $H$-module, moreover, we have the following isomorphisms
$$\Ext^i_{\QMod B}(\mathcal{N},\mathcal{M})\cong\Ext^i_{\QMod R}(\mathcal{N},\mathcal{M})^H,\ \forall\ i\ge0.$$
\end{theorem}
\begin{proof} By Lemma \ref{lem7}, the sequence (\ref{eq9}) is an injective resolution of $\mathcal{M}$ when viewed as an object in $\QMod R$. Applying $\Hom_{\QMod R}(\mathcal{N},-)$ to (\ref{eq9}), we obtain
{\tiny \begin{equation*}
  0\to \Hom_{\QMod R}(\mathcal{N},\mathcal{E}^0)\to \Hom_{\QMod R}(\mathcal{N},\mathcal{E}^1)\to\cdots\to\Hom_{\QMod R}(\mathcal{N},\mathcal{E}^i)\to\cdots.
\end{equation*}}
By Proposition \ref{prop2}, each component of the complex above is a right $H$-module and the differential is compatible with the right $H$-module structures. Taking the cohomology of the complex above, we obtain that $\Ext^i_{\QMod R}(\mathcal{N},\mathcal{M})$ is a right $H$-module for each $i\ge0$.

Now applying $\Hom_{\QMod B}(\mathcal{N},-)$ to the sequence (\ref{eq9}), we obtain the following complex
{\tiny \begin{equation}\label{eq10}
  0\to \Hom_{\QMod B}(\mathcal{N},\mathcal{E}^0)\to \Hom_{\QMod B}(\mathcal{N},\mathcal{E}^1)\to\cdots\to\Hom_{\QMod B}(\mathcal{N},\mathcal{E}^i)\to\cdots.
\end{equation}}
By Proposition \ref{prop2}, the complex (\ref{eq10}) is isomorphic to the following complex
{\tiny \begin{equation}\label{eq11}
  0\to \Hom_{\QMod R}(\mathcal{N},\mathcal{E}^0)^H\to \Hom_{\QMod R}(\mathcal{N},\mathcal{E}^1)^H\to\cdots\to\Hom_{\QMod R}(\mathcal{N},\mathcal{E}^i)^H\to\cdots.
\end{equation}}
By Lemma \ref{lem7}, (\ref{eq9}) is also an injective resolution of $\mathcal{M}$ in $\QMod R$. Noticing that the functor $(\ )^H$ commutes with taking the cohomology, we obtain that the $i$th cohomology of the complex (\ref{eq11}) is equal to $\Ext_{\QMod R}^i(\mathcal{N},\mathcal{M})^H$, which is isomorphic to $\Ext^i_{\QMod B}(\mathcal{N},\mathcal{M})$, the $i$th cohomology of (\ref{eq10}) since (\ref{eq9}) is also an injective resolution of $\mathcal{M}$ in $\QMod B$. Then we obtain the desired isomorphisms.
\end{proof}

\section{Invariant subalgebras of Hopf dense Galois extensions}\label{sec3}

Throughout this section, $H$ is a finite dimensional semisimple Hopf algebra and $R$ is a noetherian $H$-module algebra. As before, set $B:=R\#H$ and $A=R^H$.

\begin{setup}\label{setup} We assume that the $H$-module algebra $R$ satisfies the following conditions:

 (i) the $H$-module algebra $R$ is admissible and $\Tor B$ is stable (e.g. $R$ is a complete semilocal algebra, cf. Example \ref{ex1});

(ii) $R/A$ is a right $H^*$-dense Galois extension.
\end{setup}

As we have known in the beginning of Section \ref{sec2}, $R$ is a $B$-$A$-bimodule. Then we have a functor $$-\otimes_BR:\Mod B\longrightarrow \Mod A,$$ which induces a functor between the corresponding quotient categories:
$$-\otimes_{\mathcal{B}}\mathcal{R}:\QMod B\longrightarrow \QMod A.$$
We recall a result obtained in \cite{HVZ}.

\begin{lemma} \label{lem3.1} \cite[Theorem 2.4]{HVZ} Assume that Setup \ref{setup}(ii) holds. Then the functor $$-\otimes_{\mathcal{B}}\mathcal{R}:\QMod B\longrightarrow \QMod A$$ is an equivalence of abelian categories.
\end{lemma}

Let $S$ be a noetherian algebra. Recall from Section \ref{sec1} that we have a torsion functor $\Gamma_S:\Mod S\to\Mod S$. The $i$-th right derived functor of $\Gamma_S$ is denoted by $R^i\Gamma_S$. If there is no risk of confusion, we will drop the subscript $S$. Note that we have $\Gamma=\underset{\longrightarrow}\lim\Hom_S(R/K,-)$, and $R^i\Gamma=\underset{\longrightarrow}\lim\Ext^i_S(R/K,-)$, where $K$ runs over all the cofinite $S$-submodules of $S$, and the direct system is induced by the inclusion maps of cofinite submodules.

\begin{definition}\label{def3.1}
We say that a noetherian algebra $S$ is (right) {\it AS-Cohen-Macaulay} of dimension $d$ if $R^i\Gamma(S)=0$ for all $i\neq d$.

Assume $S$ is AS-Cohen-Macaulay of dimension $d$. A finitely generated right $S$-module $M$ is said to be a {\it Cohen-Macaulay} module, if $R^i\Gamma(M)=0$ for all $i\neq d$
\end{definition}

The definition of an AS-Cohen-Macaulay module is a generalization of the commutative version (cf. \cite{CH,Z}). The letters ``AS'' stand for Artin-Schelter, because the property in the definition is also a generalization of that of Artin-Schelter Gorenstein algebras (cf. \cite[Remark 8.5]{VdB}).

\begin{lemma} \label{lem3.2} Let $S$ be a noetherian algebra such that $\Tor S$ is stable. For a finitely generated right $S$-module, we have

{\rm(i)} $\Ext_{\QMod S}^i(\mathcal{S},\mathcal{M})\cong R^{i+1}\Gamma(M)$ for $i>0$, and

{\rm(ii)} an exact sequence $0\to \Gamma(M)\to M\overset{\phi}\to \Hom_{\QMod S}(\mathcal{S},\mathcal{M})\to R^1\Gamma(M)\to0$, where $\phi$ is the natural map $M=\Hom_S(S,M)\overset{\pi}\to\Hom_{\QMod S}(\mathcal{S},\mathcal{M})$ induced by the projection functor $\pi:\Mod S\to\QMod S$.
\end{lemma}
\begin{proof} For any cofinite right $S$-submodule $K$ of $S$, we have an exact sequence $0\to K\to S\to S/K\to0$, which implies isomorphisms $$\Ext^i_{S}(K,M)\cong\Ext^{i+1}_S(S/K,M),\ \forall\ i\ge1,$$ and an exact sequence
$$0\to \Hom_S(S/K,M)\to M\overset{\phi}\to \Hom_{S}(K,M)\to \Ext^1_S(S/K,M)\to0,$$ where $\phi$ is the natural map induced by the inclusion map $K\hookrightarrow S$.
Taking direct limit on all the cofinite submodules of $S$, by Lemma \ref{lem1} and $R^i\Gamma(M)=\underset{\longrightarrow}\lim\Ext^i_S(S/K,-)$ for all $i\ge0$, we obtain the desired results.
\end{proof}

\begin{theorem} \label{thm1} Assume that $R$ and $H$ satisfy the conditions in Setup \ref{setup}. If $R$ is AS-Cohen-Macaulay of dimension $d$, then so is $R^H$.
\end{theorem}
\begin{proof} Set $A:=R^H$. By Lemma \ref{lem3.1}, the functor $-\otimes_{\mathcal{B}}\mathcal{R}:\QMod B\longrightarrow \QMod A$ is an equivalence of abelian categories. Under this equivalence, the object $\mathcal{R}\in\QMod B$ corresponds to $\mathcal{A}\in \QMod A$. Then we have
\begin{equation}\label{eq3.1}
  \Ext^i_{\QMod A}(\mathcal{A},\mathcal{A})\cong \Ext^i_{\QMod B}(\mathcal{R},\mathcal{R}),\ \forall\ i\ge0.
\end{equation}
Combining (\ref{eq3.1}) with Theorem \ref{thm0}, we obtain
\begin{equation}\label{eq3.2}
  \Ext^i_{\QMod A}(\mathcal{A},\mathcal{A})\cong \Ext^i_{\QMod R}(\mathcal{R},\mathcal{R})^H,\ \forall\ i\ge0.
\end{equation}
Assume $d\ge2$. Then $\Ext^i_{\QMod R}(\mathcal{R},\mathcal{R})=0$ for all $i\neq d-1$ and $i\ge1$ by Lemma \ref{lem3.2}(i), and hence $\Ext^i_{\QMod A}(\mathcal{A},\mathcal{A})=0$ for all $i\neq d-1$ and $i\ge1$. By Lemma \ref{lem3.2}(i) again, $R^{i+1}\Gamma_A(A)=0$ for $i\neq d-1$ and $i\ge1$. Since $R^i\Gamma_R(R)=0$ for $i\neq d$, Lemma \ref{lem3.2}(ii) implies that the natural map $R\to \Hom_{\QMod R}(\mathcal{R},\mathcal{R})$ is an isomorphism. Applying the functor $(\ )^H$ on this isomorphism, we see that the natural map $A\to \Hom_{\QMod A}(\mathcal{A},\mathcal{A})$ is an isomorphism. Hence $R^i\Gamma_A(A)=0$ for $i=0,1$.

Now assume $d=1$. Similar to the above discussions, we have $R^i\Gamma_A(A)=0$ for $i\ge2$. Since $\Gamma_R(R)=0$, Lemma \ref{lem3.2}(ii) implies that the natural map $R\overset{\phi}\to \Hom_{\QMod R}(\mathcal{R},\mathcal{R})$ is injective. Applying the functor $(\ )^H$ to this map, we obtain that the natural map $A\overset{\phi}\to \Hom_{\QMod A}(\mathcal{A},\mathcal{A})$ is injective. Hence $\Gamma_A(A)=0$. The case that $d=0$ can be proved similarly.
\end{proof}

\section{Cohen-Macaulay property of $R$ as an $R^H$-module}\label{sec4}

Keep the same notions as in Section \ref{sec3}. We recall some properties of modules over a Hopf algebra.

Let $X$ be a right $H$-module. The tensor product $X\otimes H$ has two right $H$-module structures. The first one is the diagonal action of $H$, that is, $(x\otimes h)\cdot g=x\cdot g_{(1)}\otimes hg_{(2)}$ for $x\in X$, $g,h\in H$. To avoid possible confusion, the tensor product $X\otimes H$ with the diagonal action of $H$ will be denoted by $X\hat\otimes H$. The other $H$-action is the multiplication of elements of $H$ on the right of $X\otimes H$, that is, $(x\otimes h)\cdot g=x\otimes hg$ for $x\in X$ and $g,h\in H$. This right $H$-module structure will be denoted by $X\otimes H_\bullet$. The following lemma is well known.

\begin{lemma} \label{lem3.3} Let $X$ be a right $H$-module. Define maps $$\varphi:x\hat\otimes H\longrightarrow M\otimes H_\bullet,\ x\otimes h\mapsto x S(h_{(1)})\otimes h_{(2)},$$
and $$\psi:X\otimes H_\bullet\longrightarrow X\hat\otimes H,\ x\otimes h\mapsto xh_{(1)}\otimes h_{(2)}.$$
Then $\varphi$ and $\psi$ are $H$-module isomorphisms which are inverse to each other.
\end{lemma}

Let $X$ be a right $H$-module. Recall from the beginning of Section \ref{sec2} that $R$ is a right $B$-module by the right $B$-action (\ref{eq2.2}). Then $\Hom_R(X,R)$ is a right $H$-module. Note that $B$ itself is a right $B$-module. Hence $\Hom_R(X,B)$ is a right $H$-module.

\begin{lemma}\label{lem3.4} Let $X$ be a right $B$-module. Then we have an isomorphism of right $H$-modules
$$\Hom_R(X,R)\hat\otimes H\cong\Hom_R(X,R\#H).$$
\end{lemma}
\begin{proof} Firstly, note that the smash product may be written as $H\#R$ with the multiplication defined by $$(h\#r)(g\#r')=hg_{(2)}\#(S^{-1}g_{(1)}r)r'.$$ There is an isomorphism of algebras: $$\xi:H\#R\to R\#H,\ h\#r\mapsto h_{(1)}r\#h_{(2)}.$$
Define a linear map
\begin{equation*}
  \zeta:\Hom_R(X,R)\hat\otimes H\longrightarrow \Hom(X,R\#H)
\end{equation*}
by $\zeta(f\otimes h)(x)=h_{(1)}f(x)\# h_{(2)}$ for $f\in\Hom_R(X,R)$, $x\in X$ and $h\in H$. We first show that $\zeta(f\otimes h)$ is a right $R$-module morphism for every $f\in \Hom_R(X,R)$ and $h\in H$. For $r\in R$ and $x\in X$, we have
\begin{eqnarray*}
\zeta(f\otimes h)(xr)&=&h_{(1)}f(xr)\# h_{(2)}\\
&=&h_{(1)}(f(x)r)\# h_{(2)}\\
&=&(h_{(1)}f(x))(h_{(2)}r)\# h_{(3)}\\
&=&(h_{(1)}f(x)\# h_{(2)})r\\
&=&\zeta(f\otimes h)(x)r.
\end{eqnarray*}
Hence the image of $\zeta$ is indeed in $\Hom_R(X,R\#H)$. Therefore we obtain a map (still use the same notion):
\begin{equation*}
  \zeta:\Hom_R(X,R)\hat\otimes H\longrightarrow \Hom_R(X,R\#H).
\end{equation*}
We next check that $\zeta$ is a right $H$-module morphism. For $f\in \Hom_R(X,R)$, $h,g\in H$ and $x\in X$, we have
\begin{eqnarray*}
\zeta(f\otimes h)\leftharpoonup g(x)&=&\left(\zeta(f\otimes g)(xS(g_{(1)})\right)g_{(2)}\\
&=&\left(h_{(1)}f(xS(g_{(1)}))\# h_{(2)}\right)g_{(2)}\\
&=&h_{(1)}f(xS(g_{(1)}))\# h_{(2)}g_{(2)}.
\end{eqnarray*}
On the other hand, we have
\begin{eqnarray*}
\zeta\left((f\otimes h)\cdot g\right)(x)&=&\zeta(f\leftharpoonup g_{(1)}\otimes hg_{(2)})(x)\\
&=&h_{(1)}g_{(2)}f\leftharpoonup g_{(1)}(x)\# h_{(2)}g_{(3)}\\
&=&h_{(1)}g_{(3)}S^{-1}(g_{(2)})f(x S(g_{(1)}))\# h_{(2)}g_{(4)}\\
&=&h_{(1)}f(x S(g_{(1)}))\# h_{(2)}g_{(2)},
\end{eqnarray*} where the third equality follows from the definitions of the right $B$-module action (\ref{eq2.2}) on $R$ and the right $H$-action on $\Hom_R(X,R)$. Therefore, $\zeta$ is a right $H$-module morphism.

Since $H\#R$ is a finitely generated free $R$-module, it follows that $$\theta:\Hom_R(X,R)\otimes H\longrightarrow\Hom_R(X,H\#R),\ f\otimes h\mapsto [x\mapsto h\otimes f(x)]$$ is a linear isomorphism. Thus $\zeta$, equal to $\Hom_R(X,\xi)\circ\theta$, is an isomorphism since $\xi$ is an isomorphism.
\end{proof}

\begin{lemma} \label{lem3.5} Let $X$ be a right $B$-module. For each $i\ge0$, we have an $H$-module isomorphism
$$\Ext^i_R(X,R\#H)\cong \Ext_R^i(X,R)\hat\otimes H.$$ Moreover, the isomorphism is functorial in $X$.
\end{lemma}
\begin{proof} Take a projective resolution of the right $B$-module $X$ as follows:
\begin{equation}\label{eq3.3}
  \cdots\longrightarrow P^i\longrightarrow\cdots\longrightarrow P^1\longrightarrow P^0\longrightarrow X\longrightarrow0.
\end{equation}
By \cite[Proposition 2.5]{HVZ}, each $P^i$ is also projective as a right $R$-module. Apply the functor $\Hom_R(-,B)$ on (\ref{eq3.3}) to obtain the following complex
\begin{equation}\label{eq3.4}
  0\longrightarrow \Hom_R(P^0,B)\longrightarrow\Hom_R(P^1,B)\longrightarrow\cdots\longrightarrow \Hom_R(P^i,B)\longrightarrow \cdots.
\end{equation}
By Lemma \ref{lem3.4}, the complex (\ref{eq3.4}) can be written in the following way
{\small\begin{equation}\label{eq3.5}
  0\longrightarrow \Hom_R(P^0,R)\hat\otimes H\longrightarrow\Hom_R(P^1,R)\hat\otimes H\longrightarrow\cdots\longrightarrow \Hom_R(P^i,R)\hat\otimes H\longrightarrow \cdots.
\end{equation}}
Comparing the cohomology of (\ref{eq3.4}) and (\ref{eq3.5}), we obtain the desired isomorphisms. It follows from the general homology theory that the isomorphism is functorial in $X$.
\end{proof}

\begin{proposition}\label{prop4.1} Let $N$ be a finitely generated right $B$-module. Assume that the $H$-module algebra $R$ is admissible and $\Tor B$ is stable. For each $i\ge0$, we have
$$\Ext_{\QMod B}^i(\mathcal{N},\mathcal{B})\cong \Ext^i_{\QMod R}(\mathcal{N},\mathcal{R}).$$
\end{proposition}
\begin{proof} By Lemma \ref{lem1} and the assumption that the $H$-module algebra $R$ is admissible, $\Ext_{\QMod R}^i(\mathcal{N},\mathcal{B})\cong \underset{\longrightarrow}\lim\Ext^i_R(N',B)$ where $N'$ runs over all the cofinite $B$-submodules of $N$. By Lemma \ref{lem3.5}, $\Ext^i_R(N',B)\cong \Ext^i_R(N',R)\hat\otimes H$ for all cofinite $B$-submodules $N'$ of $N$. Taking the direct limits over all the cofinite $B$-submodules of $N$, we obtain
\begin{equation}\label{eq4.1}
  \Ext_{\QMod R}^i(\mathcal{N},\mathcal{B})\cong \Ext_{\QMod R}^i(\mathcal{N},\mathcal{R})\hat\otimes H.
\end{equation}
Combining Theorem \ref{thm0} and Lemma \ref{lem3.3}, we have
\begin{equation}\label{eq4.2}
  \Ext_{\QMod B}^i(\mathcal{N},\mathcal{B})\cong \left(\Ext_{\QMod R}^i(\mathcal{N},\mathcal{R})\otimes H_\bullet\right)^H.
\end{equation}
The right hand side of (\ref{eq4.2}) is isomorphic to $\Ext_{\QMod R}^i(\mathcal{N},\mathcal{R})$ as a vector space.
\end{proof}

Now we arrive at our main result of this section.
\begin{theorem} \label{thm4.1} Assume that $R$ and $H$ satisfy the conditions in Setup \ref{setup}. If $R$ is AS-Cohen-Macaulay of dimension $d$, then $R$, viewed as a right $R^H$-module, is Cohen-Macaulay (cf. Definition \ref{def3.1}).
\end{theorem}
\begin{proof} Set $A:=R^H$. Since $R/A$ is a right $H^*$-dense Galois extension, the functor $-\otimes_{\mathcal{B}}\mathcal{R}:\QMod B\to \QMod A$ is an equivalence of abelian categories (cf. Lemma \ref{lem3.1}). Note that $R$ may be viewed both as a right $A$-module and as a right $B$-module. Under the equivalence $-\otimes_{\mathcal{B}}\mathcal{R}$, $\mathcal{R}$ when viewed as an object in $\QMod A$ corresponds to $\mathcal{B}\in\QMod B$, and $\mathcal{A}\in \QMod A$ corresponds to $\mathcal{R}$ which is viewed as an object in $\QMod B$. Hence for $i\ge0$, we have
\begin{equation*}
  \Ext_{\QMod B}^i(\mathcal{R},\mathcal{B})\cong \Ext_{\QMod A}^i(\mathcal{A},\mathcal{R}).
\end{equation*}
By Proposition \ref{prop4.1}, we obtain
\begin{equation}\label{eq4.3}
  \Ext_{\QMod R}^i(\mathcal{R},\mathcal{R})\cong \Ext_{\QMod A}^i(\mathcal{A},\mathcal{R}).
\end{equation}
The rest of the proof is almost the same as that of Theorem \ref{thm1} according to the value of $d$. We omit the similar narratives.
\end{proof}

\begin{remark} Under the assumptions in Theorem \ref{thm4.1}, the isomorphism (\ref{eq4.3}) implies that
\begin{equation}\label{eq4.4}
  R^i\Gamma_{R^H}(R)\cong R^i\Gamma_R(R),\ \text{for all } i\ge0.
\end{equation}
If the algebra $R$ is a noetherian complete semilocal algebra and some further conditions are satisfied, then (\ref{eq4.4}) follows from \cite[Theorem 2.9]{WZ}.
\end{remark}

\section{Finite group actions on noetherian complete semilocal algebras} \label{sec5}

Recall that a noetherian algebra $\Lambda$ is {\it semilocal} if $\Lambda/J(\Lambda)$ is a finite dimensional semisimple algebra, where $J(\Lambda)$ is the Jacobson radical of $\Lambda$. We say that $\Lambda$ is {\it complete} if $\Lambda$ is complete with respect to the $J(\Lambda)$-adic topology, or equivalently, $\Lambda=\underset{\longleftarrow}\lim \Lambda/J(\Lambda)^n$.

Recall from \cite{WZ2} that a noetherian semilocal algebra $\Lambda$ is called a right {\it AS-Gorenstein} algebra if $\Lambda_\Lambda$ has finite injective dimension $d$, $\Ext^i_\Lambda(\Lambda/J(\Lambda),\Lambda)=0$ for $i\neq d$, and $\Ext^d_\Lambda(\Lambda/J(\Lambda),\Lambda)\cong \Lambda/J(\Lambda)$ as a left $\Lambda$-module. Left AS-Gorenstein algebras are defined similarly. If $\Lambda$ is both left and right AS-Gorenstein, then we simply say that $\Lambda$ is {\it AS-Gorenstein}. If, furthermore, $\Lambda$ has finite global dimension, then we say that $\Lambda$ is {\it AS-regular}. We remark that our definition of AS-Gorenstein algebras is a little stronger than that in \cite{WZ2}.

\begin{lemma} \label{lem5.1} Let $\Lambda$ be a noetherian semilocal algebra. Then $$\Gamma_\Lambda\cong \underset{n\to\infty}\lim\Hom_\Lambda(\Lambda/J(\Lambda)^n,-).$$
\end{lemma}
\begin{proof} Recall from Section \ref{sec1}, $\Gamma_\Lambda=\underset{\longrightarrow}\lim\Hom_\Lambda(\Lambda/K,-)$ where $K$ runs over all the cofinite right ideal of $\Lambda$. For any cofinite right ideal $K$ of $\Lambda$, there is an integer $n$ such that $(\Lambda/K)J(\Lambda)^n=0$. Then $J(\Lambda)^n\subseteq K$. Hence the direct system defined by all the cofinite right ideals of $S$ and the direct system defined by $\{J(\Lambda)^n|n\ge1\}$ are cofinal. Therefore, $\Gamma_\Lambda=\underset{\longrightarrow}\lim\Hom_\Lambda(\Lambda/K,-)\cong \underset{n\to\infty}\lim\Hom_\Lambda(\Lambda/J(\Lambda)^n,-)$.
\end{proof}

In the rest of this section, we assume $R$ that is a noetherian complete semilocal algebra with Jacobson radical $J$. Let $G\subseteq \Aut(R)$ be a finite subgroup. Then $R$ becomes a left $G$-module algebra. As before, set $B=R\#\kk G$ and $A=R^G$. Example \ref{ex1} shows that the $G$-module algebra $R$ is admissible and $\Tor B$ is stable. Note that $J$ is stable under the $G$-action. Let $\mfm=J\# \kk G$. Then $\mfm$ is a cofinite ideal of $B$, and $\mfm^n=J^n\#\kk G$ for all $n\ge1$. In general, $\mfm$ is not the Jacobson radical of $B$. Besides, for a general semisimple Hopf algebra $H$, an $H$-action on $R$ does not imply that the $H$ acts on the Jacobson radical $J$ stably. Some discussions on the stability of the $H$-action on the Jacobson radical, or on the Jacobson radical of the smash product may be found in \cite{LMS}.

Similar to the proof of Lemma \ref{lem5.1}, we also have
\begin{equation}\label{eqx5.1}
  \Gamma_B\cong \underset{n\to\infty}\lim\Hom_B(B/\mfm^n,-).
\end{equation}

Recall that the {\it depth} of a right module $M$ over a noetherian algebra $\Lambda$ is defined as follows
$$\depth_\Lambda(M)=\min\{i|R^i\Gamma_\Lambda(M)\neq 0\}.$$ We have the following Auslander-Buchsbaum formula, which is a consequence of \cite[Theorem 0.1(1)]{WZ}. We remark that our definition of the depth of a module is equivalent to that in \cite{WZ} when the algebra $\Lambda$ is a notherian complete semilocal algebra (see also, Lemma \ref{lem5.3} below).

\begin{lemma}\label{lem5.2} Let $R$ be a noetherian complete semilocal algebra which is also AS-Gorenstein of injective dimension $d$. For a finitely generated right $B$-module $M$ with finite projective dimension, we have $\pd_B(M)+\depth_B(M)=d$.
\end{lemma}
\begin{proof} Note that $M$ is also a finitely generated right $R$-module. By Lemma \ref{lemx1}, $\depth_B(M)=\depth_R(M)$. By \cite[Theorem 0.1(1)]{WZ}, we have
\begin{equation}\label{eqx5.2}
  \pd_R(M)=d-\depth_R(M)=d-\depth_B(M).
\end{equation}
Since a projective right $B$-module is projective as a right $R$-module, $\pd_R(M)\leq \pd_B(M)$. On the other hand, set $p=\pd_R(M)$. Let $\cdots \to P^i\overset{\delta^i}\to\cdots\overset{\delta^1}\to P^0\to M\to 0$ be a projective resolution of the right $B$-module $M$. Then the $p$th syzygy $\ker\delta^{p-1}$ of the resolution is projective as an $R$-module since $\pd_R(M)=p$. By \cite[Proposition 2.5]{HVZ}, $\ker\delta^{p-1}$ is projective as a right $B$-module, and hence $\pd_B(M)\leq p$. Therefore $\pd_R(M)=\pd_B(M)$. By the equalities (\ref{eqx5.2}), $\pd_B(M)+\depth_B(M)=d$.
\end{proof}

It is not hard to see that if a noetherian complete semilocal algebra is AS-Gorenstein with injective dimension $d$, then it is AS-Cohen-Macaulay of dimension $d$ in the sense of Definition \ref{def3.1} (cf. \cite{WZ}).

\begin{theorem}\label{thm5.1} Let $R$ be a noetherian complete semilocal algebra, and let $G\subseteq\Aut(R)$ be a finite subgroup. Assume that $R/R^G$ is a right $\kk G^*$-dense Galois extension and that $R$ is AS-regular of global dimension 2. Set $A=R^G$. Then the following statements hold.

{\rm(i)} $A$ is AS-Cohen-Macaulay of dimension 2.

{\rm(ii)} A finitely generated right $A$-module $M$ is Cohen-Macaulay if and only if $M\in \text{\rm add} (R_A)$, where $\text{\rm add} (R_A)$ is the subcategory of $\Mod A$ consisting of all the direct summands of finite direct sums of $R_A$.
\end{theorem}
\begin{proof} The statement (i) is a special case of Theorem \ref{thm1}.

(ii) By Theorem \ref{thm4.1}, $R_A$ is a Cohen-Macaulay module. Hence any module in $\text{add}R_A$ is Cohen-Macaulay.

Conversely, note that the functor $-\otimes_{\mathcal{B}}\mathcal{R}:\QMod B\to \QMod A$ in Lemma \ref{lem3.1} is induced by the functor $-\otimes_B R:\Mod B\to \Mod A$ (cf. \cite[Sections 2 and 3]{HVZ2}), that is, we have the following commutative diagram
\begin{equation}\label{eq5.1}
  \xymatrix{
  \Mod B \ar[d]_{\pi} \ar[r]^{-\otimes_BR} & \Mod A \ar[d]_{\pi} \\
  \QMod B \ar[r]^{-\otimes_{\mathcal{B}}\mathcal{R}} & \QMod A. }
\end{equation}
Let $F=-\otimes_BR$ and $\mathcal{F}=-\otimes_{\mathcal{B}}\mathcal{R}$. The functor $F$ has a right adjoint functor $F':=\Hom_{A^\circ}(R,-)$ (cf. \cite[Section 2]{HVZ2}), where $\Hom_{A^\circ}(-,-)$ is the Hom-functor in the category of left $A$-modules. Moreover, $FF'\cong \text{id}_{\Mod A}$. Hence $F'$ is fully faithful and $\Hom_{B}(F'X,F'Y)\overset{F}\longrightarrow\Hom_{A}(X,Y)$ is an isomorphism for any $X,Y\in\Mod A$.

Let $N_A$ be a nonzero finitely generated Cohen-Macaulay module. Then we obtain that the map $\Hom_B(F'(R),F'(N))\overset{F}\longrightarrow\Hom_A(R,N)$ is an isomorphism. Since $R/A$ is a right $\kk G^*$-dense Galois extension and $R$ is AS-regular of global dimension 2, the natural map $R\#\kk G\to \Hom_A(R,R), r\#h\mapsto [r'\mapsto r(h\cdot r')]$ is an isomorphism of algebras by \cite[Theorem 3.10]{HVZ2}. Hence as a right $B$-module, $F'(R)=\Hom_A(R,R)$ is isomorphic to $B$. Set $M=F'(N)=N\otimes _AR$.

From the above commutative diagram (\ref{eq5.1}), we obtain the following commutative diagram
\begin{equation}\label{eq5.2}
  \xymatrix{
  \Hom_B(B,M) \ar[d]_{\pi} \ar[r]^{F} & \Hom_A(R,N) \ar[d]_{\pi} \\
  \Hom_{\QMod B}(\mathcal{B},\mathcal{M}) \ar[r]^{\mathcal{F}} & \Hom_{\QMod A}(\mathcal{R},\mathcal{N}). }
\end{equation}
The top map in the diagram (\ref{eq5.2}) is an isomorphism. By Lemma \ref{lem3.1}, the bottom map is also an isomorphism. Consider the exact sequence $$0\to K\to R\to R/K\to0,$$ where $K$ is a cofinite submodule of $R$.
Applying $\Hom_A(-,N)$ to the above sequence, we obtain the following exact sequence
\begin{equation}\label{eq5.3}
 0\to \Hom_A(R/K,N)\to\Hom_A(R,N)\to\Hom_A(K,N)\to \Ext^1_A(R/K,N)
\end{equation}
Since $N$ is Cohen-Macaulay and $R/K$ is a finite dimensional, $\Hom_A(R/K,N)=0$ and $\Ext^1_A(R/K,N)=0$ by Lemma \ref{lem5.3} below.
Hence (\ref{eq5.3}) implies that the natural map $$\Hom_A(R,N)\to\Hom_A(K,N)$$ is an isomorphism for all cofinite submodules $K$ of $R$.

Taking the direct limits on both sides of the isomorphism above over all the cofinite submodules $K$ of $R$, we obtain that the projection map
\begin{equation*}
   \Hom_A(R,N)\overset{\pi}\to\Hom_{\QMod A}(\mathcal{R},\mathcal{N})
   \end{equation*}
is an isomorphism.

In summary, we obtain that the top map, the bottom map and the right vertical map in the diagram (\ref{eq5.2}) are all isomorphisms. Hence the left vertical map in the diagram (\ref{eq5.2}) is also an isomorphism.

By Lemma \ref{lem5.1}, the exact sequence $0\to \mathfrak{m}^n\to B\to B/\mathfrak{m}\to0$ implies the next exact sequence
$$0\to \Gamma_B(M)\to \Hom_B(B,M)\overset{\pi}\to \Hom_{\QMod B}(\mathcal{B},\mathcal{M})\to R^1\Gamma_B(M)\to 0.$$ Since the projection map $\pi$ is an isomorphism, $\Gamma_B(M)=R^1\Gamma_B(M)=0$. Since $R$ is of global dimension 2, the global dimension of $B$ is 2 as well (cf. \cite{L}). By Lemma \ref{lem5.2}, $M$ is a projective $B$-module. Hence $M$ is a direct summand of a free module $B^{(n)}$. Therefore $N=F(M)$ is a direct summand of $F(B)^{(n)}=(R_A)^{(n)}$, that is, $N\in \text{add} (R_A)$.
\end{proof}

\begin{lemma} \label{lem5.3} Let $A$ be a noetherian algebra. Let $M_S$ be a finitely generated module. If $R^i\Gamma_A(M)=0$ for $i<d$, then $\Ext^i_A(T,M)=0$ for all finite dimensional right $A$-module $T$.
\end{lemma}
\begin{proof} Take a minimal injective resolution of $M$ as follows
 $$0\to M\to I^0\overset{\delta^0}\to I^1\overset{\delta^1}\to\cdots\overset{\delta^{i-1}}\to I^i\overset{\delta^i}\to\cdots.$$ We claim that each $I^i$ is torsion free for all $i<d$.
Since $\Gamma_A(M)=0$, $M$ is torsion free. Hence $I^0$ is torsion free. Now suppose that there is some $i<d$ such that $I^i$ is not torsion free. Assume that $k<d$ is the smallest integer such that $\Gamma_A(I^k)\neq0$. Since the injective resolution is minimal, $\ker \delta^k$ is essential in $I^k$. Hence $\ker \delta^k\cap\Gamma_A(I^k)\neq0$. Then the kernel of the restriction $\delta^k:\Gamma_A(I^k)\to\Gamma_A(I^{k+1})$ is nontrivial. Therefore $R^k\Gamma_A(M)\neq0$, a contradiction. Hence for each $i<d$, $I^i$ is torsion free. Therefore, for any finite dimensional module $T$, $\Hom_A(T,I^i)=0$ for all $i<d$, which implies that $\Ext^i_A(T,M)=0$.
\end{proof}

\begin{remark}\label{rm5.1}  The above theorem may be viewed as a nongraded version of \cite[Section 3]{Jo}, \cite[Corollary 1.3]{Ue1} (see also \cite[Proposition 5.2]{Ue2}) and \cite[Theorem 4.4]{CKWZ}. Our method is also valid for noetherian graded algebras (cf. Example \ref{ex2}), however our method is quite different from those in \cite{Jo,Ue1,CKWZ}.
\end{remark}

\vspace{5mm}

\subsection*{Acknowledgments}
J.-W. He is supported by NSFC
(No. 11571239, 11401001)
and Y. Zhang by an FWO-grant.

\vspace{5mm}


\end{document}